\documentclass[12pt,english,british,refpage,intoc,bibliography=totoc,index=totoc,BCOR=7.5mm,captions=tableheading]{extarticle}
\usepackage{mathptmx}
\usepackage[T1]{fontenc}
\usepackage[latin9]{inputenc}
\usepackage[a4paper]{geometry}
\usepackage{color}
\usepackage{babel}
\usepackage{bm}
\usepackage{amsmath}
\usepackage{amsthm}
\usepackage{amssymb}
\usepackage[numbers]{natbib}
\usepackage[unicode=true,pdfusetitle,
 bookmarks=true,bookmarksnumbered=true,bookmarksopen=false,
 breaklinks=false,pdfborder={0 0 1},backref=false,colorlinks=true]
 {hyperref}
\hypersetup{
 pdfborderstyle=,linkcolor=black,citecolor=blue,urlcolor=black,filecolor=blue,pdfpagelayout=OneColumn,pdfnewwindow=true,pdfstartview=XYZ,plainpages=false}

\makeatletter
\theoremstyle{plain}
\newtheorem{thm}{\protect\theoremname}

\@ifundefined{date}{}{\date{}}
\usepackage{babel}
\usepackage{caption}
\usepackage[nottoc]{tocbibind}
\allowdisplaybreaks[4]
\usepackage{dsfont}
\DeclareMathAlphabet{\mathcal}{OMS}{cmsy}{m}{n}

\addto\captionsbritish{\renewcommand{\theoremname}{Theorem}}
\addto\captionsenglish{\renewcommand{\theoremname}{Theorem}}
\providecommand{\theoremname}{Theorem}

\makeatother

\addto\captionsbritish{\renewcommand{\theoremname}{Theorem}}
\addto\captionsenglish{\renewcommand{\theoremname}{Theorem}}
\providecommand{\theoremname}{Theorem}

\begin{document}
\global\long\def\Sgm{\boldsymbol{\Sigma}}%
\global\long\def\W{\boldsymbol{W}}%
\global\long\def\H{\boldsymbol{H}}%
\global\long\def\P{\mathbb{P}}%
\global\long\def\Q{\mathbb{Q}}%
\global\long\def\O{\bm{\omega}}%
\global\long\def\uu{\bm{u}}%
\global\long\def\A{\bm{A}}%
\global\long\def\S{\bm{S}}%

\title{The vortex dynamics in incompressible viscous turbulent flows }
\author{Jiawei Li\thanks{Department of Mathematical Sciences, Carnegie Mellon University, Pittsburgh,
15213. Email: jiaweil4@andrew.cmu.edu}\ \ and Zhongmin Qian\thanks{Mathematical Institute, University of Oxford, Oxford OX2 6GG. Email:
zhongmin.qian@maths.ox.ac.uk}}
\maketitle
\begin{abstract}
In this paper, we consider turbulence from a geometric perspective
based on the vorticity equations for incompressible viscous fluid
flows. We derive several quantitative statements about the statistics
of turbulent flows. In particular we establish an entropy decay inequality
which involves the turbulent kinetic energy and $L^{1}$-enstrophy,
and we identify the small time scale of the vorticity broken down
and the vorticity creation under the universality assumption of small
time scales of turbulence flows.

\selectlanguage{english}%
\medskip{}

\selectlanguage{british}%
\emph{key words}: Navier-Stokes equation, rate-of-strain tensor, Reynolds
number, turbulent flows, vorticity

\medskip{}

\end{abstract}

\section{Introduction }

The study of turbulence in fluids has been a prevalent topic as an
unfulfilled scientific discipline, although, unlike other theoretical
physics problems, the equations of motion even for turbulent flows
have been known for over a century. The study of turbulence flows
has been and will remain to be concentrated on methods of extracting
both qualitative and quantitative results about turbulent flows from
the fluid dynamic equations. The difficulty for attacking the so called
turbulence problem lies in the non-linear and non-local characteristic
of the motion equations. It has been known for quite long time that
solutions of Navier-Stokes equations are unstable under even tine
perturbations of initial data, and turbulent flows must develop into
chaotic movements if the Reynolds number becomes large. Statistical
mechanics approaches were thus developed several
decades ago, in which fluid dynamical variables were
treated as random fields, see \citep{Frisch 2004,Hinze book1975,Monin and Yaglom 1965}
for a comprehensive survey. In recent years, other ideas from field
theories and statistical mechanics such as renormalization have been
introduced to address on the other hand possible coherent structures
in chaotic turbulent motions, while only limited quantitative results
are achieved, see however \citep{McComb 1990,McComb 2014} and the
original literature therein for further information. In the last five
decades, the study of vorticity in fluids and vortex dynamics has
become an active research area in fluid mechanics and established
as an important branch of fluid dynamics. The vortex method has been
applied to the study of turbulence in particular through numerical
simulations \citep{Moffatt 1981,Moffatt and Tsinober 1992,Lesieur 1997}.
In this paper we aim to deduce a few quantitative statements on the
statistics of turbulent flows by applying the study of turbulent vortex
dynamics. We will consider incompressible viscous fluid flows in terms
of the vorticity equations, which are however equivalent to the Navier-Stokes
equations, but have an appealing feature that the vorticity actually
dominates the turbulent motion of fluids, see \citep{Foias 2001}
for a comprehensive description on the relation between Navier-Stokes
equations, vorticity and turbulence. The vorticity and the rate-of-strain
tensor have been emphasized in turbulence, and the production of the
vorticity through the rate-of-strain field has been considered as
the cause of the energy cascade, see Davidson \citep{Davidson book2015},
Hinez \citep{Hinze book1975} and Monin and Yaglom \citep{Monin and Yaglom 1965}
for its precise description.

The main results of this paper are presented in part 3 and 4, where
we consider the isotropic and non-isotropic flows respectively, and
in part 2, we will introduce several fluid dynamical variables and
derive a few equations with respect to these variables, which will
be useful in the computation later. Finally, the Einstein summation
convention is used throughout this paper.

\section{Fluid dynamical variables}

We begin with a review of several fluid dynamical variables from a
geometric perspective which may shed new light on turbulence in fluids.
The motion of a turbulent flow with viscosity $\nu>0$ may be described
by its velocity $\bm{u}=(u^{1},u^{2},u^{3})$, the fluid density $\rho$
and the pressure $p$, through the Navier-Stokes equation. There are
several fluid dynamical variables which play important roles in understanding
the underlying flow motion. Among them, the vorticity $\bm{\omega}=\nabla\times\uu$,
whose components $\omega^{i}=\varepsilon_{kji}\frac{\partial u^{j}}{\partial x^{k}}$,
where $\varepsilon$ is the Levi-Civita symbol. The total derivative
$\nabla\uu$ of $\uu$ is denoted by $\A$. Its components $A_{j}^{i}=\frac{\partial u^{i}}{\partial x^{j}}$,
are the most important fluid dynamical variables in our study below.
$\A$ is canonically decomposed into its symmetric part $\S=(S_{j}^{i})$
and its skew-symmetric part whose components are $\frac{1}{2}\varepsilon_{kji}\omega^{k}$.
In the turbulence literature, $|\bm{\omega}|^{2}$ is called the enstrophy,
$\S=(S_{j}^{i})$ is known as the strain tensor or rate-of-strain,
and $2\nu|\S|^{2}=2\nu S_{j}^{i}S_{i}^{j}$ is the energy dissipation.
These dynamical variables have clear geometric meanings. The vorticity
$\O$ is the exterior derivative of $\uu$, and the strain tensor
$\S$ is the Ricci curvature in the sense of Bakry-Emery \citep{Bakry-Emery}
associated with Taylor's diffusion of Brownian fluid particles. Therefore
the vorticity equation is exactly the Bochner identity in this context,
and the helicity $\uu\times\O$ is nothing but the Chern-Simon invariant
with respect to the material derivative $D/Dt$. Applications to isotropic
turbulence flows may be addressed therefore from this point of view.

Assume that the fluid is incompressible, hence $\nabla\cdot\uu=0$
which is equivalent to that $\textrm{tr}\A=0$. The following vector
identities 
\begin{equation}
\textrm{tr}\A^{2}=\nabla\cdot\left(\uu\cdot\nabla\uu\right)\label{tr2}
\end{equation}
and 
\begin{equation}
\textrm{tr}\A^{3}=\nabla\cdot\left[\uu\cdot\nabla\left(\uu\cdot\nabla\uu\right)-\frac{1}{2}\left(\nabla\cdot(\uu\cdot\nabla\uu)\right)\uu\right]\label{tr3}
\end{equation}
hold, which show that both $\textrm{tr}\A^{2}$ and $\textrm{tr}\A^{3}$
are exact, and $\textrm{tr}\A^{2}$, $\textrm{tr}\A^{3}$ are taken
as two parameters, named by some authors as $Q$ and $R$, in the
classification of turbulent flows, see \citep{Davidson book2015}
and \citep{Ooi Martin etc 1999} for example. Equation (\ref{tr2})
is well known and (\ref{tr3-sw}) was discovered firstly, to the best
knowledge of present authors, by Betchov \citep{Betchov 1956}. On
the other hand, a direct calculation shows that 
\begin{equation}
\textrm{tr}\A^{2}=\textrm{tr}\S^{2}-\frac{1}{2}|\O|^{2},\label{tr2-sw}
\end{equation}
\begin{equation}
\textrm{tr}\A^{3}=\textrm{tr}\S^{3}+\frac{3}{4}\O\cdot\S\O,\label{tr3-sw}
\end{equation}
and $\textrm{tr}\S^{2}$ is identified with its norm squared $|\S|^{2}$.
The following is a remarkable relation about the strain tensor $\S$
and its vorticity $\O$, which appears new to the present authors:
\begin{equation}
|\nabla\S|^{2}-\frac{1}{2}|\nabla\O|^{2}=\nabla\cdot\left[\textrm{tr}\left(\nabla\left((\uu\cdot\nabla)\nabla\uu\right)\right)-(\uu\cdot\nabla)\Delta\uu\right],\label{grad-tr sw1}
\end{equation}
where the trace on the right-hand side is taken over the two co-variant
derivatives, that is 
\[
\textrm{tr}\left(\nabla\left((\uu\cdot\nabla)\nabla\uu\right)\right)=\sum_{k}\frac{\partial}{\partial x^{k}}\left((\uu\cdot\nabla)\frac{\partial}{\partial x^{k}}\uu\right).
\]
Therefore not only $|\S|^{2}-\frac{1}{2}|\O|^{2}$ is exact, but also
is $|\nabla\S|^{2}-\frac{1}{2}|\nabla\O|^{2}$.

After having discussed a few basic vector identities about the strain
tensor and vorticity, we now want to provide a geometric interpretation
of the strain tensor. Recall that Taylor's diffusion for a viscous
fluid flow with velocity $\uu(\bm{x},t)$ is the process $X_{t}$
of Brownian fluid particles so that 
\begin{equation}
dX=\uu(X,t)dt+\sqrt{2\nu}dB,\label{eq:taylor1}
\end{equation}
where $B$ is a 3D Brownian motion. The precise meaning of stochastic
It\"{o}'s equation (\ref{eq:taylor1}) is not needed in the discussion
below, but see Ikeda-Watanabe \citep{Ikeda-Watanabe 1989} for the
theory of diffusion processes. What we need is the fact that the distribution
of the Taylor diffusion $X$ is completely determined by the backward
problem of the parabolic equation 
\[
\left(\frac{\partial}{\partial t}-(\nu\Delta+\uu\cdot\nabla)\right)f=0.
\]
In this sense one says that the elliptic operator of second order
$L=\nu\Delta+\uu\cdot\nabla$ is the infinitesimal generator of Taylor's
diffusion. It is important to notice that for incompressible fluid
the adjoint operator is given by $L^{\star}=\nu\Delta-\uu\cdot\nabla$
due to the divergence free condition $\nabla\cdot\uu=0$. According
to Bakry-Emery \citep{Bakry-Emery}, the Ricci curvature of $L$ can
be described by two step iteration: in the first iteration one recovers
the metric via the equation 
\[
\varGamma(f,g)=\frac{1}{2}\left(L(fg)-fLg-gLf\right),
\]
so that $\varGamma(f,g)=\nu\nabla f\cdot\nabla g$. The Ricci curvature
is obtained by iterating the previous process to define 
\[
\varGamma_{2}(f,g)=\frac{1}{2}\left(L\varGamma(f,g)-\varGamma(f,Lg)-\varGamma(g,Lf)\right),
\]
which yields that 
\begin{equation}
\varGamma_{2}(f,f)=\nu^{2}|\nabla^{2}f|^{2}-\nu S_{j}^{i}\frac{\partial f}{\partial x^{i}}\frac{\partial f}{\partial x^{j}}.\label{bochner1}
\end{equation}
According to the Bochner's equality \citep{Berard 1988}, one may
conclude that the Ricci curvature of $L$ is $-\S=(-S_{j}^{i})$,
and the Ricci curvature of $L^{\star}$ is the strain tensor $\S=(S_{j}^{i})$.

The Navier-Stokes equation may be written in terms of the adjoint
operator 
\begin{equation}
\left(\frac{\partial}{\partial t}-L^{\star}\right)\uu=-\nabla p,\quad\nabla\cdot\bm{u}=0,\label{eq:ns-f}
\end{equation}
and the vorticity equation is given by 
\begin{equation}
\left(\frac{\partial}{\partial t}-L^{\star}\right)\O=\S\O,\label{eq:ns-f-1}
\end{equation}
where $\S\bm{\omega}^{i}=S_{j}^{i}\omega^{j}$ for $i=1,2,3$,
which is the Weitzenb\"{o}ck formula applying to $\O$. The energy balance
equation is 
\begin{equation}
\left(\frac{\partial}{\partial t}-L^{\star}\right)|\uu|^{2}=-\nu|\O|^{2}+\nu\nabla\cdot\left(\uu\times\O\right)-2\nabla\cdot\left(p\uu\right).\label{eng-balance}
\end{equation}
In the description of turbulence flows, the enstrophy
$|\O|^{2}$ and the variation of the dissipation of energy $\textrm{tr}\S^{2}=S_{j}^{i}S_{i}^{j}$
are essential, and therefore one would like to understand their dynamics.
The evolution of the enstrophy is well known and
follows easily from the vorticity equation (\ref{eq:ns-f-1}): 
\begin{equation}
\left(\frac{\partial}{\partial t}-L^{\star}\right)\frac{|\O|^{2}}{2}=\O\cdot\S\O-\nu|\nabla\O|^{2},\label{enstropy 1}
\end{equation}
where the first term on the right-hand side is associated with the
vortex stretching. A detailed analysis of the dissipation energy $|\S|^{2}$
was initiated in the beautiful papers by Townsend \citep{Townsend 1951}
and Betchov \citep{Betchov 1956}, and there is an excellent account
in the book Davidson \citep{Davidson book2015} (pages 240 to 251).
Perhaps a better way to understand the dissipation of turbulence flows
is to start with the evolution equations for $\S=(S_{j}^{i})$ obtained
by differentiating the Navier-Stokes equation (\ref{eq:ns-f}), so
that we obtain that for all $i,j$, 
\begin{equation}
\left(\frac{\partial}{\partial t}-L^{\star}\right)S_{j}^{i}=-\sum_{k}S_{j}^{k}S_{k}^{i}+\frac{1}{4}\left(\delta_{ij}|\O|^{2}-\omega^{i}\omega^{j}\right)-\frac{\partial^{2}p}{\partial x^{i}\partial x^{j}},\label{ric-evo 1}
\end{equation}
and therefore 
\begin{equation}
\left(\frac{\partial}{\partial t}-L^{\star}\right)\textrm{tr}\S^{2}=-2\textrm{tr}\S^{3}-\frac{1}{2}\O\cdot\S\O-2\nu|\nabla\S|^{2}-2\sum_{i,j}S_{i}^{j}\frac{\partial^{2}p}{\partial x^{i}\partial x^{j}}.\label{evo-2}
\end{equation}
Now here is a non-trivial observation: the contraction between the
strain tensor $\bm{S}$ and the hessian of the pressure $p$
is in fact exact. More precisely, we have 
\[
\sum_{i,j}S_{i}^{j}\frac{\partial^{2}p}{\partial x^{i}\partial x^{j}}=\nabla\cdot\left(\uu\cdot\nabla(\nabla p)\right)-\nabla\cdot\left(\uu\Delta p\right),
\]
and thus, by combining (\ref{tr3}, \ref{tr3-sw}, \ref{grad-tr sw1})
together we obtain that 
\begin{align}
\left(\frac{\partial}{\partial t}-L^{\star}\right) & \textrm{tr}\S^{2}=\O\cdot\S\O-\nu|\nabla\O|^{2}-2\nabla\cdot\left[\uu\cdot\nabla(\nabla p)-(\Delta p)\uu\right]\nonumber \\
 & -\nabla\cdot\left[2\uu\cdot\nabla\left(\uu\cdot\nabla\uu\right)-\left(\nabla\cdot(\uu\cdot\nabla\uu)\right)\uu\right]\nonumber \\
 & -2\nu\nabla\cdot\left[\textrm{tr}\left(\nabla\left((\uu\cdot\nabla)\nabla\uu\right)\right)-(\uu\cdot\nabla)\Delta\uu\right].\label{trs2-eq}
\end{align}

Let us note that $\textrm{tr}\S=\nabla\cdot\uu=0$. Since $\S$ is
a symmetric tensor, if we assume that its three real eigenvalues are
denoted by $a,b$ and $c$, which are arranged so that $a\geq b\geq c$,
then $a+b+c=0$, $\textrm{tr}\S^{2}$ equals $a^{2}+b^{2}+c^{2}$
and $\textrm{tr}\bm{S}^{3}$ coincides with $a^{3}+b^{3}+c^{3}=3abc.$
Note that the variance of three eigenvalues $a,b$ and $c$ is a dynamical
quantity
\[
P=\frac{1}{3}\left((a-b)^{2}+(b-c)^{2}+(c-a)^{2}\right)
\]
which measures the derivation from the mean, namely from zero. Then
$P=\frac{1}{3}\textrm{tr}\S^{2}$ due to the fact that $\textrm{tr}\S=0$,
and therefore (\ref{trs2-eq}) is also the evolution equation for
the variance of the eigenvalues of the strain tensor. This identity
explains the stretching of strain tensor is proportional to that of its
magnitude.

We also note that among $a,b$ and $c$, only two of them are free.
Therefore in order to determine the evolution of these eigenvalues
of the rate-of-strain, in addition to (\ref{trs2-eq}) one needs one
more equation: naturally an evolution equation for
$\textrm{tr}\bm{S}^{3}$, which is however a bit complicated
and is given as the following: 
\begin{align*}
\left(\frac{\partial}{\partial t}-L^{\star}\right) & \textrm{tr}\bm{S}^{3}=-3|\S|^{4}+\frac{3}{4}\left(|\S|^{2}|\O|^{2}-|\S\O|^{2}\right)\\
 & -6\nu S_{k}^{i}\frac{\partial S_{i}^{j}}{\partial x^{l}}\frac{\partial S_{j}^{k}}{\partial x^{l}}-3S_{i}^{k}S_{k}^{j}\frac{\partial^{2}p}{\partial x^{i}\partial x^{j}},
\end{align*}
We will not explore this equation further in the present paper.

\section{Isotropic turbulent flows}

In this section we present several new results about homogeneous isotropic
flows \citep{Batchelor 1953,McComb 2014,Monin and Yaglom 1965} in
a developed turbulence region. Suppose $\uu(\bm{x},t)$ is the random
velocity field of an isotropic turbulent flow, so that its mean velocity
is constant (being zero without loss of generality).

The following convention is employed: if $Z$ be a dynamical variable
of turbulent flow, then $\left\langle Z\right\rangle $ denotes the
mean value of $Z$. Since the flow is isotropic, so that $\left\langle \nabla\cdot Z\right\rangle =0$
as long as $Z$ is a tensor field depending only on the turbulent
dynamical variables. Moreover, if $f$ is a dynamical scalar, then
\begin{align*}
L^{\star}f & =\nu\Delta f-\uu\cdot\nabla f\\
 & =\nu\Delta f-\nabla\cdot\left[f\uu\right]
\end{align*}
so that $\left\langle L^{\star}f\right\rangle =0$. 
\begin{thm}
For an isotropic turbulent flow $\uu(\bm{x},t)$ we have 
\begin{equation}
\left\langle |\S|^{2}\right\rangle =\frac{1}{2}\left\langle |\O|^{2}\right\rangle ,\label{mean s2}
\end{equation}
\begin{equation}
\left\langle \textrm{tr}(\S^{3})\right\rangle =-\frac{3}{4}\left\langle \O\cdot\S\O\right\rangle \label{mean s3}
\end{equation}
and 
\begin{equation}
\left\langle |\nabla\S|^{2}\right\rangle =\frac{1}{2}\left\langle |\nabla\O|^{2}\right\rangle .\label{grad-tr sw1-1}
\end{equation}
\end{thm}

\begin{proof}
The first equality (\ref{mean s2}) follows from (\ref{tr2-sw}, \ref{tr2}),
and similarly (\ref{mean s3}) is a consequence of (\ref{tr3}, \ref{tr3-sw}).
From (\ref{grad-tr sw1}) one also deduces (\ref{grad-tr sw1-1}). 
\end{proof}
For the dissipation rate of an isotropic turbulent flow, it follows
immediately from (\ref{eng-balance}) and (\ref{evo-2}) that 
\[
\frac{\partial}{\partial t}\left\langle |\uu|^{2}\right\rangle =-\nu\left\langle |\O|^{2}\right\rangle ,
\]
\[
\frac{\partial}{\partial t}\left\langle |\S|^{2}\right\rangle +\nu\left\langle |\nabla\O|^{2}\right\rangle =\left\langle \O\cdot\S\O\right\rangle .
\]

The following theorem provides a quantitative result about the energy
dissipation in isotropic turbulent flows.
\begin{thm}
For an isotropic turbulent flow then the following
entropy functional 
\begin{equation}
t\rightarrow\left\langle |\O(\cdot,t)|\right\rangle +\frac{1}{\sqrt{2}\nu}\left\langle |\uu(\cdot,t)|^{2}\right\rangle \label{enstropy}
\end{equation}
is monotonically decreasing. 
\end{thm}

\begin{proof}
If $\varPsi$ is differentiable on $(0,\infty)$, then 
\begin{equation}
\left(\frac{\partial}{\partial t}-L^{\star}\right)\varPsi(|\O|^{2})=2\varPsi'\O\cdot\S\O-2\nu\varPsi'|\nabla\O|^{2}-\nu\varPsi''\left|\nabla|\O|^{2}\right|^{2}.\label{phi eq1}
\end{equation}
Applying (\ref{phi eq1}) to $\varPsi_{\delta}(x)=\left(x+\delta\right)^{q/2}$,
where $\delta>0$ and $q\geq1$ are two constants, and using 
\[
|\nabla\O|^{2}\geq\frac{\left|\nabla|\O|^{2}\right|^{2}}{4|\O|^{2}},
\]
one obtains that 
\begin{equation}
\left(\frac{\partial}{\partial t}-L^{\star}\right)\varPsi_{\delta}(|\O|^{2})\leq-2\nu\varPsi_{\delta}'\frac{q-1}{4}\frac{\left|\nabla|\O|^{2}\right|^{2}}{|\O|^{2}+\delta}+2\varPsi_{\delta}'\O\cdot\S\O.\label{phi 2}
\end{equation}
By letting $\delta\downarrow0$ and taking mean value on both sides,
one may deduce that 
\begin{equation}
\frac{\partial}{\partial t}\left\langle |\O|^{q}\right\rangle \leq-4\left(1-\frac{1}{q}\right)\nu\left\langle \left|\nabla|\O|^{q/2}\right|^{2}\right\rangle +q\left\langle |\O|^{q-2}\O\cdot\S\O\right\rangle .\label{phi 2-1}
\end{equation}
Choosing $q=1$ we thus obtain that 
\begin{align}
\frac{\partial}{\partial t}\left\langle |\O|\right\rangle  & \leq\left\langle |\O|^{-1}\O\cdot\S\O\right\rangle \leq\sqrt{\left\langle |\S|^{2}\right\rangle }\sqrt{\left\langle |\O|^{2}\right\rangle }\nonumber \\
 & =\frac{1}{\sqrt{2}}\left\langle |\O|^{2}\right\rangle =-\frac{1}{\sqrt{2}\nu}\frac{\partial}{\partial t}\left\langle |\uu|^{2}\right\rangle ,\label{phi 2-1-1}
\end{align}
which yields that 
\[
\frac{\partial}{\partial t}\left[\left\langle |\O|\right\rangle +\frac{1}{\sqrt{2}\nu}\left\langle |\uu|^{2}\right\rangle \right]\leq0.
\]
\end{proof}

\section{Small scale of the vorticity under the similarity hypotheses }

In this section we consider turbulent flows which are not necessary
isotropic.

Consider a viscous turbulent flow (with viscosity $\nu>0$) in a region
with typical velocity $U$ being the maximum velocity, and $L$ the
typical length so that the Reynolds number is $\textrm{Re}=UL/\nu$.
As in the dimensionless analysis of fluid flows,
let us set $\bm{\varphi}(\bm{x},t)=\frac{1}{U}\uu(L\bm{x},\kappa t)$
and $\tilde{p}(\bm{x},t)=\frac{1}{U^{2}}p(L\bm{x},\kappa t)$,
where $\kappa=L/U$. Then the Navier-Stokes equation turns into 
\begin{equation}
\left(\frac{\partial}{\partial t}+\bm{\varphi}\cdot\nabla-\frac{1}{\textrm{Re}}\Delta\right)\bm{\varphi}=-\nabla\tilde{p},\quad\nabla\cdot\bm{\varphi}=0.\label{sim01}
\end{equation}
Let $\bm{\theta}=\nabla\times\bm{\varphi}$ and $\varGamma_{ij}=\frac{1}{2}\left(\frac{\partial\varphi_{i}}{\partial x_{j}}+\frac{\partial\varphi_{j}}{\partial x_{i}}\right)$,
$1\leq i,j\leq3$, so that $\bm{\theta}(\bm{x},t)=\frac{L}{U}\bm{\omega}(L\bm{x},\kappa t)$,
$\bm{\varGamma}(\bm{x},t)=\frac{L}{U}\bm{S}(L\bm{x},\kappa t)$.
The dimensionless vorticity equation then is given by 
\begin{equation}
\left(\frac{\partial}{\partial t}+\bm{\varphi}\cdot\nabla-\frac{1}{\textrm{Re}}\Delta\right)\bm{\theta}=\bm{\varGamma}\bm{\theta}.\label{vort-norm1}
\end{equation}
The enstrophy equation is written as the following dimensionless form:
\begin{equation}
\left(\frac{\partial}{\partial t}+\bm{\varphi}\cdot\nabla-\frac{1}{\textrm{Re}}\Delta\right)\frac{|\bm{\theta}|^{2}}{2}=\bm{\theta}\cdot\bm{\varGamma}\bm{\theta}-\frac{1}{\textrm{Re}}|\nabla\bm{\theta}|^{2}.\label{enstrop2}
\end{equation}

The similarity hypothesis claims that, in the developed turbulent
region, the local small structures of turbulent flows with the same
Reynolds number are the same statistically. A trivial solution to
(\ref{sim01}) with the maximum velocity is the the constant solution
$\bm{\varphi}=(1,1,1)$, the fundamental solution associated
with the corresponding elliptic operator $\frac{1}{\textrm{Re}}\Delta\pm\nabla$
is given by 
\[
\varGamma_{\pm}(\bm{\xi},t,\bm{x})=\sigma^{3}\left(\frac{1}{2\pi t}\right)^{3/2}\exp\left[-\frac{\left|\sigma(\bm{x}-\bm{\xi})\mp\sigma t\right|^{2}}{2t}\right]
\]
for $t>0$, where $\sigma=\sqrt{\textrm{Re}/2}$ for simplicity. For
a general $\bm{\varphi}$, there are explicit bounds for the
fundamental solution $\varGamma(s,\bm{\xi},t,\bm{x})$
associated with $\frac{1}{\textrm{Re}}\Delta-\bm{\varphi}\cdot\nabla$
obtained in \citep{Qian and Zheng 2004}. For $\beta\in\mathbb{R}$,
let 
\begin{align*}
p^{\beta}(x,t,y) & =\frac{1}{\sqrt{2\pi t}}\int_{|x-y|/\sqrt{t}}^{\infty}ze^{-\left(z-\beta\sqrt{t}\right)^{2}/2}dz\\
 & =e^{-\frac{\beta^{2}}{2}t+\beta|x-y|}\frac{1}{\sqrt{2\pi t}}e^{-\frac{|x-y|^{2}}{2t}}+\beta\Psi\left(\frac{|x-y|-\beta t}{\sqrt{t}}\right),
\end{align*}
where $\Phi(a)=\frac{1}{\sqrt{2\pi}}\int_{-\infty}^{a}e^{-z^{2}/2}dz$
and $\Psi(a)=1-\Phi(a)$. Then we have 
\[
\sigma^{3}\prod_{i=1}^{3}p^{\sigma}(\sigma\xi_{i},t-\tau,\sigma x_{i})\leq\varGamma(\tau,\bm{\xi},t,\bm{x})\leq\sigma^{3}\prod_{i=1}^{3}p^{-\sigma}(\sigma\xi_{i},t-\tau,\sigma x_{i}).
\]

Let $\delta=t-\tau>0$ be the elapsing time. Then 
\begin{align*}
f_{\pm}(\sigma,\xi,\delta,x) & \equiv\sigma p^{\pm\sigma}(\sigma\xi,\delta,\sigma x)\\
 & =e^{-\frac{\sigma^{2}}{2}\delta\pm\sigma^{2}|x-y|}\frac{1}{\sqrt{2\pi\delta\sigma^{-2}}}e^{-\frac{|x-y|^{2}}{2\delta\sigma^{-2}}}\pm\sigma\Psi\left(\frac{|x-y|\mp\delta}{\sqrt{\delta\sigma^{-2}}}\right).
\end{align*}

Now observe that, when $\delta>0$ is small and $\sigma>0$ is large,
it follows that 
\begin{align*}
f_{\pm}(\sigma,\xi,\delta,x) & \sim e^{-\frac{\sigma^{2}}{2}\delta\pm\sigma^{2}|x-y|}\delta_{x}(d\xi)\pm\sigma\Psi\left(\frac{|x-y|\mp\delta}{\sqrt{\delta\sigma^{-2}}}\right)\\
 & \sim\left\{ \begin{array}{ccc}
e^{-\frac{\sigma^{2}}{2}\delta}\delta_{x}(d\xi) & , & \textrm{ if \ensuremath{|x-y|\mp\delta>0};}\\
e^{-\frac{\sigma^{2}}{2}\delta}\delta_{x}(d\xi)\pm\sigma & , & \textrm{ if \ensuremath{|x-y|\mp\delta<0};}\\
e^{-\frac{\sigma^{2}}{2}\delta}\delta_{x}(d\xi)\pm\frac{\sigma}{2} & , & if\ensuremath{|x-y|\mp\delta=0}.
\end{array}\right.
\end{align*}
In particular, 
\[
f_{-}(\sigma,\xi,\delta,x)\sim e^{-\frac{\sigma^{2}}{2}\delta}\delta_{x}(d\xi),
\]
and 
\[
f_{+}(\sigma,\xi,\delta,x)\sim\left\{ \begin{array}{ccc}
e^{-\frac{\sigma^{2}}{2}\delta}\delta_{x}(d\xi) & , & \textrm{ if \ensuremath{|x-y|>\delta};}\\
e^{-\frac{\sigma^{2}}{2}\delta}\delta_{x}(d\xi)+\sigma & , & \textrm{ if \ensuremath{|x-y|<\delta};}\\
e^{-\frac{\sigma^{2}}{2}\delta}\delta_{x}(d\xi)+\frac{\sigma}{2} & , & if\ensuremath{|x-y|=\delta}.
\end{array}\right.
\]

Now we apply this to the vorticity equation (\ref{vort-norm1}). We
are interested in the enstrophy transfer in a small
scale. Therefore we assume that the tensor-of-strain $\bm{\varGamma}$
is a constant symmetric matrix with three eigenvalues $a\geq b\geq c$
with $a+b+c=0$. Then 
\[
\bm{\theta}(\bm{x},t)=e^{(t-\tau)\bm{\varGamma}}\int\varGamma(\tau,\bm{\xi},t,\bm{x})\bm{\theta}(\bm{\xi},\tau)d\xi
\]
for $\tau<t$ with $\delta=t-\tau>0$ being small.

Since $e^{\delta\bm{\varGamma}}\approx\bm{I}+\delta\bm{\varGamma}$
, we obtain that 
\[
\theta^{i}(\bm{x},t)\sim\int\varGamma(\tau,\bm{\xi},t,\bm{x})\theta^{i}(\bm{\xi},\tau)d\xi+\delta\int\varGamma(\tau,\bm{\xi},t,\bm{x})\Gamma_{j}^{i}\theta^{j}(\bm{\xi},\tau)d\xi.
\]
Let us replace $\varGamma(\tau,\bm{\xi},t,\bm{x})$
by its bounds, and obtain 
\begin{align*}
\theta^{i}(\bm{x},t) & \sim e^{-3\frac{\sigma^{2}}{2}\delta}\theta^{i}(\bm{x},\tau)+e^{-3\frac{\sigma^{2}}{2}\delta}\delta\Gamma_{j}^{i}\theta^{j}(\bm{x},\tau)\\
 & +\sigma\int_{|\xi-x|<\delta}\theta^{i}(\bm{\xi},\tau)d\xi+\sigma\delta\int_{|\xi-x|<\delta}\Gamma_{j}^{i}\theta^{j}(\bm{\xi},\tau)d\xi.
\end{align*}
Therefore 
\[
\theta^{i}\left(x,t\right)\sim e^{-3\frac{\sigma^{2}}{2}\delta}\theta^{i}(\bm{x},\tau)+e^{-3\frac{\sigma^{2}}{2}\delta}\delta\Gamma_{j}^{i}\theta^{j}(\bm{x},\tau)+\sigma\int_{|\xi-x|<\delta}\theta^{i}(\bm{\xi},\tau)d\xi.
\]

This equation shows that the original vorticity in a turbulent flow
will be quickly destroyed after $\kappa\delta\gg\frac{L}{U}\frac{2}{\textrm{Re}}=\frac{2\nu}{U^{2}}$
and new vorticity created about the time $\kappa\delta\sim\frac{L}{U}\frac{2}{\textrm{Re}}=\frac{2\nu}{U^{2}}$,
which are independent of the size of the turbulent flows as one may
expect.

\end{document}